\documentclass[12pt]{article}

\usepackage{amsmath}
\usepackage{amssymb}
\usepackage[english]{babel}
\usepackage{hyperref}

\newtheorem{thm}{Theorem}

\newtheorem{lem}[thm]{Lemma}

\newtheorem{claim}[thm]{Claim}
\newenvironment{proof}{{\bf Proof.} }{$\Box$\\}

\newcommand{\CD}{\mathrm{CD}}
\newcommand{\IPC}{\mathrm{IPC}}

\newcommand{\StwoS}{\mathrm{S2S}}

\newcommand{\PROP}{\mathrm{PROP}}
\newcommand{\intt}{\rm{int}}
\newcommand{\vp}{\varphi}
\newcommand{\imp}{\Rightarrow}

\newcommand{\cP}{\mathcal{P}}
\newcommand{\cT}{\mathcal{T}}
\newcommand{\cO}{\mathcal{O}}

\newcommand{\conj}{\wedge}
\newcommand{\disj}{\vee}
\newcommand{\mbb}[1]{{\mathbb{#1}}}

\newcommand{\set}[1]{\left\{#1\right\}}

\newcommand{\bR}{{\mbb{R}}}

\newcommand{\isfunc}[3]{#1\colon #2 \longrightarrow #3}

\newcommand{\bQ}{\mathbb{Q}}
\newcommand{\impl}{\Rightarrow}
\newcommand{\veps}{\varepsilon}

\newcommand{\Closed}{\mathrm{Closed}}

\newcommand{\clBelong}{\mathrm{clBelong}}

\newcommand{\sekrul}[4]{
\frac{ #1\vdash #2}{#3\vdash #4}
}

\newcommand{\dsekrul}[6]{
\frac{ #1\vdash #2\ \ ||\ \ #3 \vdash #4}{#5\vdash #6}
}

\newcommand{\tsekrul}[8]{
\frac{ #1\vdash #2\ \ ||\ \ #3 \vdash #4\ \ ||\ \ #5 \vdash #6}{#7\vdash #8}
}

\usepackage{stmaryrd}

\newcommand{\ldbrack}{\llbracket}
\newcommand{\rdbrack}{\rrbracket}

\title{Second order intuitionistic propositional logic 
of the real  line is decidable\footnote{Partially supported by MNiSW Grant N N206 355836}}
\author{Konrad Zdanowski\\ University of Cardinal Stefan Wyszy{\'n}ski, Warsaw}
\date{}
\begin{document}
\maketitle
\begin{abstract}
It is known that the set of tautologies of second order intuitionistic 
propositional logic, $\IPC 2$, is undecidable.
Here, we prove that the sets of formulas of $\IPC 2$ which are
true in the algebra of open subsets of  reals or rationals 
are decidable.
\end{abstract}

\section{Basic definitions}

We investigate the second order intuitionistic propositional logic, denoted as $\IPC 2$ 
(for a detailed treatment of this logic we refer to  the book \cite{SU}).
The set of formulas
of this logic is the same as in the classical case. We have standard propositional
connectives, universal and existential quantifiers and the set of propositions. Firstly, we present 
the set of axioms and rules in the Gentzen style. Then, we define various semantics for $\IPC 2$
and describe their status with respect to completeness and decidability of the tautology problem. 

Below, $\Gamma$ is a  multiset of formulas of $\IPC 2$, $\psi$,
$\vp$ and $\rho$ are formulas of $\IPC 2$ and $p$ is a proposition.
The letters I and E in names of rules stand for the   ``introduction''  
and ``elimination'', respectively.
\begin{enumerate}
\item Axioms:
$$
\Gamma, \psi \vdash \psi.
$$
\item Rules for conjunction:
  \begin{align*}
    \dsekrul{\Gamma}{\psi}{\Gamma }{\vp}{\Gamma}{\psi\land\vp} 
\,\textrm{($\land$I)}, & & 
    \sekrul{\Gamma}{\vp\land \psi}{\Gamma}{\vp},\ \ \ 
    \sekrul{\Gamma}{\vp\land \psi}{\Gamma}{\psi} \,\textrm{($\land$E)}.
  \end{align*}
\item Rules for disjunction:
  \begin{align*}
    \sekrul{\Gamma}{\vp}{\Gamma}{\vp\lor\psi}, \ \  \
    \sekrul{\Gamma}{\psi}{\Gamma}{\vp\lor\psi}\,\textrm{($\lor$I)}, & & 
\tsekrul{\Gamma,\vp}{\rho}{\Gamma,\psi}{\rho}{\Gamma}{\vp\lor\psi}{\Gamma}{\rho}
   \,\textrm{($\lor$E).}
  \end{align*}
\item Rules for implication:
  \begin{align*}
    \sekrul{\Gamma,\vp}{\psi}{\Gamma}{\vp\impl \psi}\, \textrm{($\impl$I)}, & &
    \dsekrul{\Gamma}{\vp\impl\psi}{\Gamma}{\vp}{\Gamma}{\psi}\,\textrm{($\impl$E).} 
  \end{align*}
\item The rule \emph{ex falso quodlibet}:
\begin{align*}
\sekrul{\Gamma}{\bot}{\Gamma}{\vp}\,\textrm{($\bot$E)}.
\end{align*}
\item Rules for quantifiers:
  \begin{align*}
    \sekrul{\Gamma}{\vp}{\Gamma}{\forall p\, \vp}\,\textrm{($\forall$I)}, & & 
    \sekrul{\Gamma}{\forall p\, \vp}{\Gamma}{\vp[p:=\psi]}\,\textrm{($\forall$E)},
  \end{align*}
  \begin{align*}
    \sekrul{\Gamma}{\vp[p:=\psi]}{\Gamma}{\exists p\, \vp}\,\textrm{($\exists$I)}, & & 
    \dsekrul{\Gamma}{\exists p\, \vp}{\Gamma,\vp}{\psi}{\Gamma}{\psi}
        \,\textrm{($\exists$E)}. 
  \end{align*}
In rules ($\forall$I) and ($\exists$E)  we have a restriction  that the variable
$p$ should not occur as a free variable of $\Gamma$ or $\psi$
\end{enumerate}
We denote the above calculus with $\IPC 2$ as well. 
Later, we consider sets of tautologies of $\IPC 2$
for various kinds of algebraic semantics. 
However, these sets will contain the set of theorems of the above
calculus with one exception of $\IPC 2^{-}$ (defined later). 
By $\IPC$ we denote intuitionistic propositional logic (without quantification)
and the corresponding calculus defined by removing from the above one the rules
for quantifiers.

In the above formulation we have no special rules and even no symbol for negation. 
This is so, because we do not treat negation as a primitive but rather we define 
$\neg\vp$ as $\vp\impl\bot$. We define also $\vp\Leftrightarrow \psi$ as 
$(\vp\imp \psi)\land (\psi\imp \vp)$. It is known that in the given calculus one can 
define from $\forall$ and $\imp$ all other connectives and quantifiers. So, the following 
formulas are provable, 
\begin{align*}
\bot & \Leftrightarrow  \forall p\, p\\
\vp \lor \psi & \Leftrightarrow \forall p((\vp\imp p)\imp((\psi\imp p)\imp p)),\\
\vp \land \psi & \Leftrightarrow \forall p ((\vp\imp (\psi \imp p))\imp p),\\
\exists q \vp(q) & \Leftrightarrow \forall p(\forall q(\vp(q)\imp p)\imp p).\\\
\end{align*}
In the proof of our main theorem we use this fact implicitly by restricting our 
translation to formulas with $\forall$ and $\imp$, only. Let us mention that 
we do need $\forall$ quantifier to define other connectives, see e.g.~\cite{SU} or \cite{SU10}.

It was proved by L\"{o}b (see \cite{L76}) that the above calculus has an undecidable provability problem.  
Even the $\forall$--free fragment of this logic is undecidable (see \cite{SU10}).
Before L\"{o}b's article, Gabbay in \cite{G74} considered  
$\IPC 2$ extended by a scheme called the axiom of constant domains ($\CD$),
$$
\forall p(\vp\lor \psi(p))\imp (\vp\lor \forall p \psi(p)),
$$
where $p$ is not free in $\vp$. In the context of first order 
intuitionistic logic this scheme was introduced by Grzegorczyk in~\cite{G64} and it is also
called Grzegorczyk scheme (one should not confuse this with Grzegorczyk axiom
in modal logic). Gabbay showed undecidability of ($\IPC 2 + \CD$), see~\cite{G74},
but his proof was later corrected by Sobolev in~\cite{S77}. Gabbay claimed that his result
generalizes to the case without $\CD$. According to Gabbay, the generalization
could be obtained by the finite axiomatizability of $\CD$ over  $\IPC 2$.
Nevertheless, it seems that there is no obvious method to define such an axiomatization.  

Sobolev also considered logics without full
comprehension axioms which correspond in our setting to rules $\forall E$ and $\exists I$. 
In the restricted versions of both rules we demand that the formula $\psi$ is atomic.
Let us call
this logic $\IPC 2^-$.  Sobolev showed then that any logic between $\IPC 2^-$ and ($\IPC 2 + \CD$)
is undecidable.

Now, we will discuss  various semantics for $\IPC 2$. There are two kinds
of popular semantics for intuitionistic logics, one constructed using Kripke models
and the other one based on Heyting algebras. In the Kripkean approach semantics 
is given by Kripke frames $(C,\leq, \{D_c\colon c\in C\})$, where
each $D_c$  is a subset of $\cP(C)$ and all $X\in D_c$ are upward closed 
with respect to $\leq$. Moreover, for each $c\leq c'\in C$ we have $D_c\subseteq D_{c'}$.
The set $C$ is the set of possible worlds and, for each $c\in C$, the set $D_c$ is the 
range of second order quantification in the world $c$. Then, the value of a given formula $\vp$
as the set of possible worlds at which $\vp$ is true
may be given by a usual inductive definition. In order to satisfy unrestricted versions
of rules $\forall E$ and $\exists I$ one needs to require that for any formula $\vp$, any $c\in C$ 
and  for any valuation $v$ of propositions into $D_c$, if $X_{\vp, v}$ 
is the set of worlds greater or equal $c$ at which
$\vp$ is satisfied with $v$, then $X_{\vp, v}\in D_c$, for each $c\in C$. 
This class of models gives a sound and complete semantics for $\IPC 2$. If we drop
the last condition concerning rules $\forall E$ and $\exists I$ then we get
a sound and complete semantics for $\IPC 2^{-}$.
The logics with  $\CD$ are complete w.r.t. ``constant domains semantics'' where all $D_c$'s are 
equal, see~\cite{G74}. Then,  the range for quantifiers is the same in all possible worlds. 
Let us stress that $D_c$'s may not contain all upward closed subsets of $C$. 
Adding the condition that $D_c$'s  contain all upward closed subsets of $C$ gives us
the, so called, \emph{principal} semantics. 
The set of tautologies of principal Kripkean semantics is recursively isomorphic 
to the classical second order logic as proved by
Kremer, see~\cite{K97a}.

Now, we turn to algebraic semantics. In the classical case the sound and complete semantics
for second order propositional logic is given by boolean algebras. In the intuitionistic case the sound
and complete semantics for $\IPC$ is given by Heyting algebras.
 
A Heyting algebra $(H,\leq, \cap,\cup,\rightarrow,0,1)$ is a distributive lattice with top and bottom elements  augmented with the pseudo-complement operation $\rightarrow$ which interprets implication. It is required that the following is well defined for $a,b\in H$,
$$
a\rightarrow b =\max\{ c\in H\colon c\cap a \leq b\}.
$$
In the case of $\IPC$, Heyting algebras give the sound and
complete semantics, e.g., by constructing a Heyting algebra from a Kripke model where all upward closed sets from the Kripke model form the universe of the algebra. 

In the algebraic semantics we have two ways in which we can interpret quantification.
In the, so called, principal semantics quantifiers range over all elements of a given algebra
and their meaning is given by infinite joins and meets, which have to exist. In the non-principal
semantics quantifiers range over a distinguished subset  of an algebra.
In the second order case the relation between Kripkean and algebraic semantics is not 
as straightforward as in the quantifier free case. An obvious way to translate a Kripke frame $(C,\leq, \dots)$
into a special kind of Heyting algebra, a topology, would be to define 
a topology of all upward closed subsets of $C$.   
However, not all upward closed subsets of a Kripke model
are within its domain of quantification. Moreover, for different possible worlds $c\in C$ 
we may have different domains $D_c$. 
Only if we consider principal Kripkean semantics then for a given
frame $(C, \leq, \dots)$ we may define a topology
of all upward closed subsets of $C$ (see, e.g., Exercise 2.8 in \cite{SU}).
For such topology the satisfaction relation is preserved
from the one of the Kripke frame. 

Lately, Philip Kramer in a personal communication
expressed his strong confidence that the complexity of the set of tautologies
for principal algebraic semantics is as hard as in the case of principal Kripkean
semantics. Kramer made his statement in his article \cite{K97b} (p. 296)
claiming that a nontrivial extension of methods from \cite{K97a}
would be needed. In a non-principal case a recent article by Kramer, \cite{K13}, establishes
its completeness w.r.t. $\IPC 2$.

A special case of Heyting algebras is given by topologies. Let us describe a satisfiability relation for $\IPC 2$
and topological principal semantics.
Let $\cT=(T,\cO(T))$ be an arbitrary topology, where $\cO(T)$ is the set of open subsets of $T$, 
and let $\isfunc{v}{\PROP}{\cO(T)}$ be a valuation
from the set of propositions. Then, we may define a value of a given formula of $\IPC 2$ in $\cT$ under 
$v$, denoted as $\ldbrack\vp\rdbrack^{\cT}_v$,  by recursion on the complexity of the formula:
\begin{enumerate}
\item $\ldbrack \bot\rdbrack^{\cT}_v = 0$,
\item $\ldbrack p\rdbrack^{\cT}_v = v(p)$,
\item $\ldbrack \psi\conj \gamma \rdbrack^{\cT}_v 
  = \ldbrack \psi \rdbrack^{\cT}_v \cap\ldbrack \gamma \rdbrack^{\cT}_v$,
\item $\ldbrack \psi\disj \gamma \rdbrack^{\cT}_v 
  = \ldbrack \psi \rdbrack^{\cT}_v \cup\ldbrack \gamma \rdbrack^{\cT}_v$,
\item $\ldbrack \psi\impl \gamma \rdbrack^{\cT}_v 
  = \intt(( T\setminus \ldbrack \psi \rdbrack^{\cT}_v) \cup \ldbrack \gamma \rdbrack^{\cT}_v)$,
\item $\ldbrack \exists p \psi \rdbrack^{\cT}_v 
      =  \bigcup_{a\in \cO(T)}\ldbrack \psi \rdbrack^{\cT}_{v(p\mapsto a)})$,
\item $\ldbrack \forall p \psi \rdbrack^{\cT}_v 
  = \intt( \bigcap_{a\in \cO(T)}\ldbrack \psi \rdbrack^{\cT}_{v(p\mapsto a)})$.
\end{enumerate}
We say that  $\vp$ is true in $\cT$ under $v$ if $\ldbrack \vp\rdbrack^{\cT}_v = T$.

It is known that the topologies of open subsets of $\bR$ or $\bQ$ form a sound and complete
semantics for $\IPC$. It can be easily shown that it is not the case for $\IPC 2$. 
Indeed, for each two $r, r'\in \bR$ there is a homeomorphism of $\bR$ 
into itself mapping $r$ to $r'$. It follows that if we have a sentence $\psi$
of $\IPC 2$ then either 
$\ldbrack \psi\rdbrack^{\bR}_v= \bR $ or $\ldbrack \psi \rdbrack^{\bR}_v= \emptyset$
(note that the value of $\ldbrack \psi\rdbrack^{\bR}_v$ does not depend on $v$, for a sentence $\psi$).
Therefore, for each sentence $\psi$ of $\IPC 2$, $\psi \lor \neg \psi$ is true in $\bR$. 
Of course, some sentences of that form are not provable in $\IPC 2$.
One can check that if we take an arbitrary quantifier free formula $\vp(p_1, \dots, p_n)$ which is a classical 
tautology and is not an intuitionistic tautology then for $\psi:=\forall p_1\dots \forall p_n \vp$, the formula
$\psi\lor \neg \psi$ is not provable in $\IPC 2$.
The very same argument works also for $\bQ$. 

On the other hand the sentence $\neg\forall p(p\lor \neg p)$ is true in $\bR$ (and in $\bQ$) though it 
is not valid intuitionistically  and moreover it is a classical contrtautology.
It follows that the $\IPC 2$ theory of $\bR$ or $\bQ$ is not a subset of 
classical tautologies.

Despite the above facts, the topologies of the real and rational lines are natural semantics
for intuitionistic logics. Firstly, it is natural to ask about second order propositional 
theory of these models which are  kind of standard models for $\IPC$. 
Secondly, one can see $\IPC 2$ over $\bR$ or $\bQ$ as a language capable of expressing some 
properties of these topologies. Thus, we may ask about the decidability of topological theories 
of $\bR$ or $\bQ$ expressible in $\IPC 2$. 

We show here, that $\IPC 2$ tautologies of principal semantics of reals or rationals  
are easier than in the general case, namely decidable.
The method used in the proof is an interpretation of these theories into the monadic theory 
of infinite binary tree, proved to be decidable by  Rabin's result (for details 
on this subject we refer to \cite{GWT02}).

Let $T^\omega=\set{0,1}^*$ be the set of finite binary sequences.
The infinite binary tree is a structure $\cT^\omega=(T^\omega, s_0, s_1, \leq)$, where  
$s_0(u)=u0$ and $s_1(u)=u1$ and $u \leq v$ when $u$ is an initial segment of $v$. 
A path in $\cT^\omega$ is an infinite set $P\subseteq T^\omega$ such that $P$ is closed 
on initial segments and is linearly ordered by~$\leq$. The empty sequence is denoted by $\varepsilon$.

The monadic second order logic is an extension of first order logic by second order quantifiers ranging
over subsets of a given universe. Rabin's theorem states that the monadic second order theory
of $\cT^\omega$ is decidable. We will denote this theory by $\StwoS$.  
It should be noted that the complexity of $\StwoS$ is non-elementary.
It became a standard method to show decidability of various problems by reducing them
to $\StwoS$. In the next section, we exhibit a reduction for theories 
of $\IPC 2$ of reals and rationals.

\section{Decidability on $\bR$ and $\bQ$}
\subsection{Interpretation in $\textrm{S2S}$}

We give interpretations of $\IPC 2$ theories of $\cO(\bR)$  and  $\cO(\bQ)$ in $\StwoS$. 
A similar though a bit simpler
 interpretation was used in \cite{Z04} showing decidability of $\IPC 2$ (and S4 with 
 propositional quantification) on trees of height and arity $\leq \omega$ (in the principal
 semantics).

\begin{thm}
The $\IPC 2$ theories of the open subsets of reals and the open subsets
of rational numbers are interpretable in  $\StwoS$.
\end{thm}
\begin{proof}
Let us recall that we write $s_0(x)$ and $s_1(x)$ to denote 
respectively the left and  the right successors, and $x \leq y$ to denote
that $x$ is on the path from the root of the tree to $y$.

 Firstly, we give an interpretation of  the $\IPC 2$ theory of reals.
Instead of thinking about $\bR$ we take an open interval $(0,1)$ which has the same topological
properties. In particular any topological operation is taken in $(0,1)$, e.g., the closure of 
$(0,1\slash 3)$ is $(0,1\slash 3]$. 

Each real number $r\in(0,1)$ may be seen as its binary representation $0.a_0a_1a_2\dots$,
 where $a_i\in\set{0,1}$ and $r=\Sigma_{i=0}^{\infty} a_i 2^{-i-1}$.
Such representations can be interpreted as infinite paths in $T^\omega$.
A binary sequence $0.a_0a_1a_2\dots$ is therefore a path 
$\set{\varepsilon, s_{a_0}(\veps), s_{a_1}s_{a_0}(\veps), \dots}$.
In what follows we will use both representations: infinite $\set{0,1}$--sequences
$a_0a_1a_2\dots$ (without the leading ``$0$.'')  
and paths in $T^\omega$.  
In order to have the unique representation of each real we exclude
sequences which have only finitely many zeros.

We can define on such infinite paths in $T^\omega$ 
the topology inherited from $(0,1)$.
 What we need is to assure that formulas of $\IPC 2$ 
can be effectively translated into formulas of $\StwoS$ such that 
they will be equivalent modulo the translation of open sets of both topologies.

An infinite binary path represents a real from $(0,1)$ if and only if it is not 
of the form $0^\omega$ or $u1^\omega$ for some $u\in\set{0,1}^*$.
The set of such paths is obviously  definable in $\StwoS$.
A formula $\textrm{Path}(X)$ stating that $X$ is an infinite path is just
$$
(X(x) \land X(y) \imp (x\leq y \lor y\leq x)) \land
$$
$$
\forall x (X(x) \imp (X(s_0(x)) \lor X(s_1(x)))) \land X(\varepsilon).
$$
Then a formula $U(X)$  defining the set of paths which represent some real 
can be written as:
$$
\textrm{Path}(X)\land \forall x\big(X(x)\imp \exists z \geq x\; X(s_0(z))\big)\land \exists x\, X(s_1(x)).
$$
The  first conjunct of the above formula  states that $X$ is an infinite path, the second one
states that there are infinitely many $0$'s in $X$ and the third conjunct states that there is at least one $1$ in $X$. 
For a set $X$ such that $\textrm{Path}(X)$,  let $r(X)$ be a real represented by $X$.

We need to  represent not only  real numbers but also  
subsets of $(0,1)$. For a subset $S \subseteq T^\omega$, by
$R(S)$ we denote a set of reals such that their corresponding paths are contained 
in $S$,
$$
R(S)=\set{r(X) \colon  U(X) \land X\subseteq S  }.
$$

We will represent open subsets of $(0,1)$ by their closed complements. For a set $S\subseteq T^\omega$,
$R(S)$ is closed in $(0,1)$ if the following formula, $\Closed(S)$, holds:
\begin{eqnarray*}
\lefteqn{\forall X\subseteq S\{[\textrm{Path}(X)\land \exists y(X(s_0(y)) \land  \forall z\geq s_0(y)\, \neg X(s_0(z))))] \imp  }\\
          &&   \exists Y\subseteq S\exists y[\textrm{Path}(Y)\land Y(y)\land X(y)\land \\
          &&   \ \ \ \ \ \ \ \ \ \ \ \ \ \   X(s_0(y)) \land \forall z\geq s_0(y)( X(z)\imp X(s_1(z)))  \land \\
         && \ \ \ \ \ \ \ \ \ \ \ \ \ \  Y(s_1(y)) \land \forall z \geq s_1(y) (Y(z)\imp Y(s_0(z)))]\}.
\end{eqnarray*}
The formula above states that if a path $X$ of the form $u01^\omega$ is a subset of $S$ then there 
is a path $Y\subseteq S$ of the form $u10^\omega$. The condition is necessary because  we do not allow paths of the form $u1^\omega$ to represent reals (and satisfy the predicate $U(X)$). 
Thus, if we have such a path 
$X\subseteq S$, then we require that $S$ contains also a path $Y$ such that $r(X)=r(Y)$ and $U(Y)$.
Otherwise, it could happen that for some sequence of sets 
$\{X_i\subseteq S\colon i\in\omega\land U(X_i)\}$ such that $\lim_{i\rightarrow\infty} r(X_i)=r(X)$
(in fact all $r(X_i)$ may be less than $r(X)$) there is no $Y\subseteq S$ such that $U(Y)$ and $r(Y)=r(X)$.

We use two facts about sets satisfying formula $\Closed(S)$. 
\begin{claim}
\begin{enumerate}
\item
For any $C$ closed in $(0,1)$ there exists $S\subseteq T^\omega$ such that 
$\Closed(S)$ and $R(S)=C$.
\item
For any $S\subseteq T^\omega$ such that $\Closed(S)$, $R(S)$ is closed in $(0,1)$.
\end{enumerate}
\end{claim}
\begin{proof}
To show $1$ it is enough to take $S=\bigcup\{X\subseteq T^\omega\colon U(X)\land r(X)\in C\}$.
Then,  $\Closed(S)$ holds. Indeed, let a path $u01^\omega$ be a subset 
of $S$, then there is a sequence of reals $r_i\in C$, for $i\in \omega$ such that $r_i=r(u01^{i}v_i)$, 
for  infinite binary words $v_i$ where all paths $u01^{i}v_i$ 
are subsets  of $C$. 
The sequence $r_i$ converges to a real $r(u10^\omega)$
and since $C$ is closed $r(u01^\omega)\in C$ and so $u10^\omega$ is a subset of $S$.   

Obviously, $C\subseteq R(S)$. To prove the converse let us assume that $r=r(X)\in R(S)$ 
for some $X\subseteq R(S)$ such that $U(X)$. Let us assume towards a contradiction that $r\not\in C$.
We have two cases to consider. The first one when $X$ is of the form $u10^\omega$ and 
the second, complementary case. We consider the latter. Then, for each $i\in \omega$ 
there is $r_i\in C$ and $X_i\subseteq S$ such that $U(X_i)$, $r_i=r(X_i)$ and $X$ has a common initial segment
 with $X_i$ of length $i$. This is so because any element of $S$ belongs to a path representing 
 a real from $C$. Now, $r=\lim_{i\rightarrow \infty} r_i$ and, since $C$ is closed, $r\in C$, a contradiction.
As for the case of $X=u10^\omega$ we repeat the same reasoning either with $X$ or 
with a path $u01^\omega$. In both cases we get the same contradiction $r(X)=r(u01^\omega)\in C$.

To show $2$  let $r_i\in R(S)$ be a sequence of reals converging to some $r\in(0,1)$. 
Let $P_i\subseteq S$ be such that $U(P_i)$ and $r_i=r(P_i)$ and let $P\subseteq T^\omega$
be such that $U(P)$ and  $r=r(U)$.
If $P_i$ are of the form $u01^{n_i}v_i$, for some strictly increasing sequence $n_i$, then $P$
is a path $u10^\omega$ and, 
by $\Closed(S)$, $P\subseteq S$. It follows that $r \in R(S)$.
Otherwise, $P$ is a path with infinitely many $1$'s and $r=\sum_{i\in\omega} 2^{-n_i}$,
for some strictly increasing sequence $n_i$. Now, if $|r-r_i| < 2^{-n_i-2}$ then $P$ and $P_i$
have a common initial segment of length $n_i$. 
We obtain that
$P\subseteq \bigcup_{i\in\omega} P_i\subseteq S$ and, therefore, $r\in R(S)$.
\end{proof}

The above claim shows that sets of the form $\Closed(S)$ are a good representation
of closed subsets of $(0,1)$. We can write an $\StwoS$ formula  $\clBelong(X,S)$ expressing that a real $r(X)$ belongs to a closed set $R(S)$.
It has the form
$$
U(X)\land \Closed(S)\land X\subseteq S.
$$
Similarly, we can express that a closed set $R(S)$ is included in a set $R(T)$ with
$$
\forall X(\clBelong(X,S)\imp \clBelong(X,T)).
$$

Let us state a useful lemma about definability in $\StwoS$. 
\begin{lem}\label{lem:minmax}
For each $\StwoS$ formula $\vp(X)$ with $X$ a free second order variable and possibly with some first
and second order parameters there exists a formula  $\min_\vp(X)$ such that
\begin{itemize}
\item
 if there exists a unique minimal closed set $C\subseteq (0,1)$ such that 
$\vp(X)$ is true for any $X$ with $C=R(X)$,  
then $\min_\vp(X)$ is true only about sets $X$ satisfying $C=R(X)$, 
\item
$\min_\vp(X)$ if false for any set $X$, otherwise.
\end{itemize}

\end{lem}
\begin{proof}
We write a formula $\min_\vp(X)$ as 
$$
\vp(X)\land \Closed(X)\land
$$
$$
\forall Y((\Closed(Y)\land \vp(Y))\imp \forall Z( (U(Z)\land \clBelong(Z,X))\imp \clBelong(Z,Y))).
$$
\end{proof}

Now,we define an inductive translation of an $\IPC 2$ formula $\vp(p_1,\dots,p_n)$  
into an $\StwoS$ formula  $\vp^*(T,T_1,\dots,T_n)$. We represent open sets by its closed complements. 
We require the following property:
for all open subsets $R, R_1,\dots, R_n$ of $(0,1)$ and 
all $X,X_1,\dots,X_n\subseteq T^\omega$ such that $\Closed(X)$, $R=(0,1)\setminus R(X)$ 
and $\Closed(X_i)$, $R_i=(0,1)\setminus R(X_i)$, for  $i\leq n$, we have the equivalence, 
$$
[\vp]^{(0,1)}_{\set{p_i\mapsto R_i}}=R \textrm{\ \ if and only if\ \ } 
$$
$$
  (\set{0,1}^*,s_0,s_1, \leq)\models \vp^*[X, X_1,\dots,X_n].
$$

If $\vp=\bot$, then $\vp^*=\forall x  T(x)$ (note that if $X=T^\omega$
then $R(X)=(0,1)$ and we want the complement of $X$ to be the empty set). 
If $\vp=p_i$, then 
$\vp^*= \forall Y (U(Y)\imp (\clBelong(Y,T) \Leftrightarrow \clBelong(Y,T_i)))$.

For $\vp=(\psi_1\imp \psi_2)$, we have
\begin{eqnarray*}
[\vp]^{(0,1)}_v&=& \intt\left(((0,1)\setminus [\psi_1]^{(0,1)}_v)\cup [\psi_2]^{(0,1)}_v\right)\\
&=&\max\{O\subseteq (0,1)\colon O\textrm{ is open } \land O\subseteq ((0,1)\setminus [\psi_1]^{(0,1)}_v)\cup [\psi_2]^{(0,1)}_v \}\\
&=& (0,1)\setminus \min\{ C\subseteq (0,1)\colon C\textrm{ is closed } \land \\
   && \ \ \ \ \ \ \ \ \ \ \ \ \ \ \ \ \ \ \ \ \ \ \ \ \ \ \ \ \ \ \  ([\psi_1]^{(0,1)}_v\cap
                     ((0,1)\setminus [\psi_2]^{(0,1)}_v))\subseteq C\}.\\
\end{eqnarray*}
By properties of the topology the above maximum and minimum exist. We need to write
a formula $\psi^*(T,T_1,\dots, T_n)$ such that with parameters $X_1,\dots, X_n$ substituted for 
$T_1,\dots,T_n$, respectively, it will be true only about the unique $T$ with
$$
R(T)=C_0=\min\{ C\subseteq (0,1)\colon C\textrm{ is closed } \land 
            ([\psi_1]^{(0,1)}_v\cap ((0,1)\setminus [\psi_2]^{(0,1)}_v))\subseteq C\}.
$$
Let $\widehat{\vp}(T,T_1,\dots,T_n)$ be a formula 
$$
\Closed(T)\land \psi^*_1(T^*_1,T_1,\dots,T_n)\land \psi^*_2(T^*_2,T_1,\dots,T_n)\land
$$
$$
\forall X((U(X)\land \neg \clBelong(X, T^*_1)\land \clBelong(X,T^*_2))\imp \clBelong(X,T)).
$$
The formula above expresses the definitional property of $C_0$ in $\StwoS$ and the topology of $T^\omega$
inherited from $(0,1)$. Now, as $\vp^*(T,T_1,\dots,T_n)$ we  take 
the formula $\min_{\widehat{\vp}(T,\dots)}(T,T_1,\dots,T_n)$ from Lemma~\ref{lem:minmax} where the minimum is taken over $T$. The formula $\vp^*(T,\dots)$ is true only about the set $C_0$ what 
proves the inductive thesis for $\vp$.

If $\vp=\forall p_n \psi(p_n)$ then 
\begin{eqnarray*}
[\vp]^{(0,1)}_v & = & 
            \intt(\bigcap_{O\textrm{ is open }} [\psi]^{(0,1)}_{v(p\mapsto O)})\\
&=& \max\{ S\subseteq (0,1)\colon S\textrm{ is open and for all open $O\subseteq (0,1)$, }
                                                      S\subseteq [\psi]^{(0,1)}_{v(p\mapsto O)} \}\\ 
&=& (0,1)\setminus \min\{C\subseteq (0,1)\colon C\textrm{ is closed and }\\
&& \textrm{\ \ \ \ \ \ \ \ \ \ \  \ \ \ \ \ \ \ \ \ \ \ \ \ \ \ \ \  for all open $O\subseteq (0,1)$, }
                                                    (0,1)\setminus [\psi]^{(0,1)}_{v(p\mapsto O)}\subseteq C\}.
\end{eqnarray*}
Since any topology is a complete Heyting algebras,  the above sets are well defined. 
The last expression can be translated to an $\StwoS$ formula. Let $\widehat{\vp}(T)$ be the following
formula
\begin{eqnarray*}
  \lefteqn{\Closed(T)\land}\\
  \lefteqn{ \forall W\forall T_n [(\Closed(W)\land \Closed(T_n)\land \psi^*(W,T_1,\dots,T_n))\imp}\\
   &&\ \ \ \ \ \ \ \ \ \ \ \ \  \ \  \ \ \ \ \forall Y ((U(Y)\land \clBelong(Y,W)) \imp \clBelong(Y,T))].
\end{eqnarray*}
Now, using Lemma \ref{lem:minmax}, we can write $\vp^*(T)$ as 
$\min_{\widehat{\vp}(T)}(T,T_1,\dots,T_{n-1})$.

The above translation gives us decidability of $\IPC 2$ on $(0,1)$ since for any
$\IPC 2$ sentence $\vp$, 
$$
\textrm{$\vp$ is true in $(0,1)$ if and only if}
$$
$$ 
\textrm{$\forall T \forall X((\Closed(T)\land \vp^*(T) \land U(X)) \imp \neg \clBelong(X,T))$
is true in $\cT^\omega$.}
$$

A similar procedure gives also decidability of the $\IPC 2$ theory  of open subsets of rationals.
One needs to use the fact that the topology of dyadic rationals from $(0,1)$
is isomorphic to the topology of $\bQ$. 
Then, the set of paths which correspond to 
these rationals is easily definable as paths of the form $u10^\omega$, 
for some $u\in \set{0,1}^*$. Now, let $U_\bQ(X)$ be a formula which defines
these paths. In order to obtain an interpretation of the $\IPC 2$ theory of open subsets of rationals
one should restrict the universe of the given above interpretation to infinite paths
satisfying $U_\bQ$. Syntactically, one should replace each occurrence of $U(X)$
with $U_\bQ(X)$.
\end{proof}

The above reduction gives a non elementary upper bound on the complexity of $\IPC 2$
on reals or rationals. We conjecture that the complexity of  these theories is in fact elementary.
\bigskip

\noindent
{\bf Acknowledgments.} I would like to thank Pawe{\l} Urzyczyn for urging me to write 
this paper.


\begin{thebibliography}{GWT02}


\bibitem[G74]{G74}
Dov M. Gabbay,
On 2nd Order Intuitionistic Propositional Calculus with Full Comprehension,
 Archiv f\"{u}r mathematische Logik und Grundlagenforschung, 16(1974), 
pp.~177--186.

\bibitem[GWT02]{GWT02}
Automata, Logics, and Infinite Games, eds.
Erich Gr\"{a}del,  Wolfgang Thomas,  Thomas Wilke, Lecture Notes in Computer Science, vol. 2500, 
Springer, 2002.  

\bibitem[G64]{G64}
Andrzej Grzegorczyk,
A Philosophically Plausible Formal Interpretation of Intuitionistic Logic,
Indagationes Mathematicae, 26(1964), pp.~596--601.

\bibitem[K97a]{K97a}
Philip Kremer,
On the Complexity of Propositional Quantification in Intuitionistic Logic,
Journal of Symbolic Logic, 62(1997), pp.~529--544.

\bibitem[K97b]{K97b}
Philip Kremer,
Propositional Quantification in the Topological Semantics for S4,
Notre Dame Journal of Formal Logic, 38(1997), pp.~295--313.

\bibitem[K13]{K13}
Philip Kremer,
Completeness of second-order propositional
S4 and H in topological semantics,
available at 
\url{http://individual.utoronto.ca/philipkremer/onlinepapers.html}.

\bibitem[L76]{L76}
M. H. L\"{o}b, 
Embedding First Order Predicate Logic in Fragments of Intuitionistic Logic,
Journal of Symbolic Logic,
41(1976), pp.~705--718.

\bibitem[S77]{S77}
S. K. Sobolev,
The Intuitionistic Propositional Calculus with Quantifiers,
Matematicheskie Zametki, 22(1977), pp.~69--76. English translation
in  Mathematical Notes of the Academy of Sciences of the USSR, 22(1977), 
pp.~528--532, doi:     10.1007\slash BF01147694.

\bibitem[SU]{SU}
Morten Heine S{\o}rensen,  Pawe{\l} Urzyczyn, Lectures on the Curry-Howard Isomorphism,
Elsevier, 2006.


\bibitem[SU10]{SU10}
Morten Heine S{\o}rensen, Pawe{\l} Urzyczyn,
A Syntactic Embedding of Predicate Logic into Second-Order Propositional Logic,
Notre Dame J. Formal Logic,  51(2010), pp.~457--473. 

\bibitem[TFHN]{TFHN}
Makoto Tatsuta, Ken-etsu Fujita, Ryu Hasegawa, and Hiroshi Nakano, 
Inhabitation of Polymorphic and Existential Types, Annals of Pure and Applied Logic, 161(2010),
pp.~1390--1399.


\bibitem[Z04]{Z04}
Richard Zach,
Decidability of quantified propositional intuitionistic logic and S4 on trees of height and arity $\leq\omega$,
Journal of Philosophical Logic, 33 (2004), pp.~155--164.

\bibitem[Z09]{Z09}
Konrad Zdanowski,
On the  second order intuitionistic propositional logic without a universal quantifier, 
Journal of Symbolic Logic, 74(2009), pp.~157--167.

\end{thebibliography}

\end{document}